\documentclass{article}

\usepackage{fullpage,amsthm,amsmath,amssymb,amsbsy,amsfonts}

\title{Optimal estimates for the gradient of harmonic functions
in the multidimensional half-space}
\author{Gershon Kresin and  Vladimir Maz'ya\\ \\
{\it Dedicated to Louis Nirenberg on the occasion of his eighty fifth birthday}}

{ \date\ }

\numberwithin{equation}{section}
\newtheorem{lemma}{Lemma}
\newtheorem{theorem}{Theorem}
\newtheorem{proposition}[theorem]{Proposition}
\newtheorem{corollary}{Corollary}
\newcommand{\bs}{\boldsymbol}
\newenvironment{remark}{{\bf Remark}}

\begin{document}
\maketitle
{\bf Abstract.}
A representation of the sharp constant in a pointwise estimate of the gradient
of a harmonic function in a multidimensional half-space is obtained under the assumption that
function's boundary values belong to  $L^p$. This representation is concretized for the
cases $p=1, 2,$ and $\infty$.
\\
\\
{\bf 2000 MSC.} Primary: 35B30; Secondary: 35J05
\\
\\
{\bf Keywords:} Real-part theorems, multidimensional harmonic functions,
estimates of the gradient, Khavinson's problem

\setcounter{equation}{0}
\section{Introduction}
There is a series of sharp estimates for the first derivative of a function $f$ analytic
in the upper half-plane ${\mathbb C}_{+}=\{ z \in {\mathbb C}: \Im z > 0 \}$ with different
characteristics of the real part of $f$ in the majorant part  (see, e.g., \cite{KM}).
We mention, in particular, the Lindel\"of inequality in the half-plane
\begin{equation} \label{E_KM1}
| f'(z) | \leq \frac{1}{\Im z}\;
\sup_{\zeta \in {\mathbb C}_{+}} \{ \Re f(\zeta ) -\Re f(z) \}
\end{equation}
and two equivalent inequalities
\begin{equation} \label{E_KM2}
\hspace{-12mm}| f'(z) | \leq \frac{ 2}{\pi \Im z}\;
\sup _{\zeta \in {\mathbb C}_{+}}|\Re f(\zeta )|
\end{equation}
and
\begin{equation} \label{E_KM2A}
\hspace{-14mm}| f'(z) | \leq \frac{1}{\pi \Im z}\;\;
{\rm osc}_{_{\;{\mathbb C}_{+}}}\!(\Re f),
\end{equation}
where ${\rm osc}_{_{\;{\mathbb C}_{+}}}\!(\Re f)$ is the oscillation
of $\Re f$ on ${\mathbb C}_{+}$, and $z$ is an arbitrary point in
${\mathbb C}_{+}$.

Inequalities for analytic functions with certain characteristics of its real part
as majorants, are called {\it real-part theorems} in reference to the first
assertion of such a kind, the celebrated Hadamard real-part theorem
$$
|f(z)|\leq \frac{C|z|}{1-|z|}\;\max_{|\zeta|=1}\Re f(\zeta).
$$
Here $|z|<1$ and $f$ is an analytic function on the closure ${\overline {\mathbb D}}$
of the  unit disk ${\mathbb D}=\{ z: |z|< 1 \}$  vanishing at $z=0$. This inequality was
first obtained by Hadamard with $C=4$ in 1892 \cite{Ha}. The following refinement of
Hadamard real-part theorem due to Borel \cite{Bo1,Bo2}, Carath\'eodory \cite{Lan1,Lan2}
and Lindel\"of \cite{Lin}
\begin{equation} \label{HBC}
|f(z)-f(0)|\leq \frac{2|z|}{1-|z|} \sup_{|\zeta|<1}\Re \left \{ f(\zeta) - f(0) \right \},
\end{equation}
and  corollaries of the last sharp estimate are often called the
Borel-Carath\'eodory inequalities. Sometimes, (\ref{HBC}) is called
Hadamard-Borel-Carath\'eodory inequality (see, e.g. Burckel \cite{Bur}).
The collection of  real-part theorems and related assertions is rather broad. It
involves assertions of various form (see, e.g. \cite{KM} and the
bibliography collected there).

Obviously, the inequalities for the first derivative
of an analytic function (\ref{E_KM1})-(\ref{E_KM2A})
can be restated as estimates for the gradient of a harmonic function. For example,
inequality (\ref{E_KM2}) can be written in the form
\begin{equation} \label{L_H}
| \nabla u(z) | \leq \frac{2}{\pi y}\sup_{\zeta \in {\mathbb R}_{+}^2 } |u(\zeta)|,
\end{equation}
where  $u$ is a harmonic function in the half-plane
${\mathbb R}_{+}^2=\{ z=(x, y) \in {\mathbb R}^2 : y > 0 \}$.

\smallskip
In the present work we find a representation for the sharp coefficient
${\mathcal C}_p( x)$ in the inequality
\begin{equation} \label{E_IN1}
\left |\nabla u(x)\right |\leq {\mathcal C}_p( x)\big|\!\big |u \big |\!\big |_p,
\end{equation}
where $u$ is harmonic function in the half-space
${\mathbb R}^{n} _{+}=\big \lbrace x=(x', x_n): x'=(x_1,\dots,x_{n-1})
\in {\mathbb R}^{n-1}, x_{n} > 0\big \rbrace$, represented by the Poisson integral with
boundary values in $L_p({\mathbb R}^{n-1})$,  $||\cdot ||_p$ is the norm in
$L_p({\mathbb R}^{n-1})$, $1\leq p \leq \infty$, $x \in {\mathbb R}^{n} _{+}$. It is shown that
$$
{\mathcal C}_p( x)=C_p\;x_n^{(1-n-p)/p}
$$
and explicit formulas for $C_p$ in (\ref{E_IN1}) for $p=1, 2, \infty$ are given.

Note that a direct consequence of (\ref{E_IN1}) is the following sharp limit
relation for the gradient of a harmonic function in the $n$-dimensional domain
$\Omega $ with smooth boundary:
$$
\lim _{x \rightarrow {\cal O}_x}\;\big |x - {\cal O}_x \big |^{(n+p-1)/p}\;
\sup \big \{ |\nabla u(x)|: ||u |_{\partial\Omega}||_p \leq 1 \big \}=C_p,
$$
where ${\cal O}_x$ is a point at $\partial\Omega$ nearest to $x \in \Omega$ (compare
with Theorem 2 in \cite{KM_1}, where a relation of the same nature for the values
of solutions to elliptic systems was obtained).

\smallskip
In Section 2 we characterize $C_p$ in terms of an extremal problem on the unit
hemisphere in ${\mathbb R}^{n}$.
In Section 3 we reduce this problem to that of finding of the supremum
of a certain double integral, depending on a scalar parameter and show that
$$
C_1=\frac{2 (n-1)}{\omega _n}={(n-1)\Gamma\left ( n/2 \right )
\over \pi ^{n/2}}\;,\;\;\;\;\;\;\;\;\;\;\;
C_2=\sqrt{\frac{n(n-1)}{2^n\omega_{n}}}=
\sqrt{(n-1) \Gamma \left ( {n+2 \over 2}\right ) \over 2^n \pi ^{n/2}}\;,
$$
where $\omega_{n}$ is the area of the unit sphere in ${\mathbb R}^{n}$.

In Section 5 we treat the more difficult case of $p=\infty$.
We anticipate the proof of the main result by deriving in Section 4
an algebraic inequality to be used for finding an explicit formula for
$C_\infty $. Solving the variational problem stated in Section 3,
we find
$$
C_\infty=\frac{4(n-1)^{(n-1)/2}\;\omega_{n-1}}{ n^{n/2}\;\omega_n}
={4(n-1)^{(n-1)/2}\Gamma \left ( {n \over 2}\right ) \over \sqrt{\pi}n^{n/2}
\Gamma \left ( {n-1 \over 2}\right )}\;.
$$
In particular,
$$
C_\infty={4 \over 3\sqrt{3}}\;,\;\;\;\;\;\;
C_\infty={3\sqrt{3} \over 2\pi}\;.
$$
for $n=3$ and $n=4$, respectively.

As a trivial corollary of the inequality
\begin{equation} \label{E_IN2}
|\nabla u(x)|\leq \frac{4(n-1)^{(n-1)/2}\;\omega_{n-1}}{ n^{n/2}\;\omega_n x_n}\;
\sup _{y \in {\mathbb R}^{n} _{+}}|u(y)|\;,
\end{equation}
which is equivalent to (\ref{E_IN1}) with $p=\infty $,
we find
\begin{equation} \label{E_IN3}
|\nabla u(x)|\leq \frac{2(n-1)^{(n-1)/2}\;\omega_{n-1}}{ n^{n/2}\;\omega_n x_n}\;
{\rm osc}_{_{{\mathbb R}^{n}_+}}\!(u),
\end{equation}
where ${\rm osc}_{_{{\mathbb R}^{n}_+}}\!(u)$
is the oscillation
of $u$ on ${\mathbb R}^{n} _{+}$. The sharp inequalities (\ref{E_IN2}) and (\ref{E_IN3})
are multidimensional generalizations of analogues of the real-part theorems (\ref{E_KM2})
and (\ref{E_KM2A}), respectively.

The sharp constant in the inequality
$$
\left |{\partial u \over \partial |x|}\Big |_{x_{_0}} \right |
\leq K(x_{_0})\sup _{|y|<1}|u(y)|,
$$
where $u$ is a harmonic function in the three-dimensional unit ball $B$ and
$x_{_0} \in B$, was found by Khavinson \cite{KHAV}, who suggested,
in a private conversation, that the same constant $K(x_{_0})$ should appear in
the stronger inequality
$$
|\nabla  u(x_{_0})| \leq K(x_{_0})\sup _{|y|<1}|u(y)|.
$$
When dealing in Section 4 with the analogue of Khavinson's problem
for the multidimensional half-space, we show that in fact, the sharp constants
in pointwise estimates for the absolute value of the normal derivative and of
the modulus of the gradient of a harmonic function coincide. We also show that
similar assertions hold for  $p=1$ and $p=2$.
\setcounter{section}{1}
\setcounter{equation}{0}
\section{Auxilliary assertion}

We introduce some notation used henceforth.
Let ${\mathbb R}^{n} _{+}=\big \lbrace x=(x', x_n): x'=(x_1,\dots,x_{n-1})
\in {\mathbb R}^{n-1}, x_{n} > 0\big \rbrace$, ${\mathbb S}^{n-1}=
\{ x \in {\mathbb R}^{n}: |x|=1 \}$,
${\mathbb S}^{n-1}_+= \{ x \in {\mathbb R}^{n}: |x|=1,\; x_n>0 \}$ and
${\mathbb S}^{n-1}_{-}= \{ x \in {\mathbb R}^{n}: |x|=1,\; x_n<0 \}$.
Let $\bs e_\sigma $ stand for the
$n$-dimensional unit vector joining the origin to a point $\sigma $
on the sphere ${\mathbb S}^{n-1}$.

By $||\cdot ||_p$ we denote the norm in
the space $L^p({\mathbb R}^{n-1})$, that is
$$
|| f||_p=\left \{\int _ {{\mathbb R}^{n-1} } |f(x')|^p \;dx' \right \}^{1/p},
$$
if $1\leq  p< \infty $, and $|| f||_\infty =\mbox{ess}\;\sup \{ | f(x') |:
x' \in {\mathbb R}^{n-1} \}$.

Next, by $h^p({\mathbb R}^n_{+})$ we denote  the Hardy space of harmonic functions
on ${\mathbb R}^n_{+}$, which can be represented as the Poisson integral
\begin{equation} \label{EH_1}
u(x)=\frac{2}{\omega _n}\int _{{\mathbb R}^{n-1}}\frac{x_n}{|y-x|^n}\;u(y')dy'
\end{equation}
with boundary values in $L^p({\mathbb R}^{n-1} )$, $1\leq  p\leq \infty $,
where $y=(y', 0)$, $y' \in {\mathbb R}^{n-1}$.

Now, we find a representation for the best coefficient ${\mathcal C}_p( x; z)$
in the inequality for the absolute value of derivative of $u$ at $x\in {\Bbb R}^n_+$ in the
arbitrary direction $\bs z \in {\Bbb S}^{n-1}$ assuming that $L_p({\Bbb R}^{n-1})$.
In particular, we obtain a formula for the constant in a similar inequality for the
modulus of the gradient.

\setcounter{theorem}{0}
\begin{lemma} \label{L1_1}
Let $u \in h^p({\mathbb R}^n_{+})$, and let  $x $ be  an arbitrary point in
${\mathbb R}^n _+$. The sharp coefficient ${\mathcal C}_p (x; \bs z)$ in the inequality
$$
|\left ( \nabla u(x), \bs z \right ) |\leq {\mathcal C}_p( x; \bs z)
\big|\!\big |u \big |\!\big |_p
$$
is given by
\begin{equation} \label{EH_2A}
{\mathcal C}_p( x; \bs z)=C_p(\bs z) x_n^{(1-n-p)/p},
\end{equation}
where
\begin{equation} \label{EH_2B}
C_1(\bs z)=\frac{2}{\omega _n}\sup _{\sigma \in {\mathbb S}^{n-1}_+ }\big |\big (\bs e_n -n(\bs e_{\sigma}, \bs e_n)\bs e_{\sigma},\; \bs z \big )\big |\big (\bs e_{\sigma}, \bs e_n \big )^n,
\end{equation}
\begin{equation} \label{EH_3DA}
C_p(\bs z)=\frac{2}{\omega _n}
\left \{ \int _ {{\mathbb S}^{n-1}_+ }
\big |\big (\bs e_n -n(\bs e_{\sigma}, \bs e_n)\bs e_{\sigma},\; \bs z \big )\big |^{p/(p-1)}
\big (\bs e_{\sigma}, \bs e_n \big )^{n/(p-1)}\;d\sigma \right \}^{(p-1)/p}
\end{equation}
for $1<p<\infty $,
and
\begin{equation} \label{EH_3H}
C_\infty(\bs z)=\frac{2}{\omega _n}
\int _ {{\mathbb S}^{n-1}_+ }
\big |\big (\bs e_n -n(\bs e_{\sigma}, \bs e_n)\bs e_{\sigma},\; \bs z \big )\big |
\;d\sigma .
\end{equation}

In particular, the sharp coefficient ${\mathcal C}_p (x)$ in the inequality
$$
\left |\nabla u(x)\right |\leq {\mathcal C}_p( x)\big|\!\big |u \big |\!\big |_p
$$
is given by
\begin{equation} \label{EH_3J}
{\mathcal C}_p(x)=C_p\; x_n^{(1-n-p)/p},
\end{equation}
where
\begin{equation} \label{EH_4J}
C_p=\sup _{|\bs z|=1}C_p(\bs z).
\end{equation}
\end{lemma}
\begin{proof} Let $x=(x', x_n)$ be a fixed point in ${\mathbb R}^n_+$.
The representation (\ref{EH_1}) implies
$$
\frac{\partial u}{\partial x_i}=\frac{2}{\omega _n}\int _{{\mathbb R}^{n-1}}
\left [ \frac{\delta_{ni}}{|y-x|^n} + \frac{nx_n (y_i-x_i)}{|y-x|^{n+2}}  \right ]u(y')dy',
$$
that is
\begin{eqnarray*}
\nabla u(x)&=&\frac{2}{\omega _n}\int _{{\mathbb R}^{n-1}}
\left [\; \frac{\bs e_{n}}{|y-x|^n} + \frac{nx_n (y-x)}{|y-x|^{n+2}}\; \right ]u(y')dy'\\
& &\\
&=&\frac{2}{\omega _n}\int _{{\mathbb R}^{n-1}}
\frac{\bs e_{n}  - n(\bs e_{xy}, \bs e_{n})\bs e_{xy}}{|y-x|^n}\; u(y')dy',
\end{eqnarray*}
where $\bs e_{xy}=(y-x)|y-x|^{-1}$.
For any $\bs z \in {\mathbb S}^{n-1}$,
\begin{equation} \label{EH_DIR}
(\nabla u(x), \bs z)=
\frac{2}{\omega _n}\int _{{\mathbb R}^{n-1}}
\frac{(\bs e_{n}  - n(\bs e_{xy}, \bs e_{n})\bs e_{xy},\; \bs z)}{|y-x|^n}\; u(y')dy'.
\end{equation}
Hence,
\begin{equation} \label{EH_7AD}
{\mathcal C}_1( x; \bs z)=\frac{2}{\omega _n}
\sup _{y \in {\mathbb R}^{n-1}}
\frac{|(\bs e_{n}  - n(\bs e_{xy}, \bs e_{n})\bs e_{xy},\; \bs z)|}{|y-x|^n},
\end{equation}
and
\begin{equation} \label{EH_6A}
{\mathcal C}_p( x; \bs z)=\frac{2}{\omega _n}
\left \{\int _{{\mathbb R}^{n-1}}
\frac{\big |\big (\bs e_{n}  - n(\bs e_{xy}, \bs e_{n})\bs e_{xy}, \bs z \big )\big |^q }
{|y-x|^{nq}}\;dy' \right \}^{1/q}
\end{equation}
for $1<p \leq \infty $, where $p^{-1}+q^{-1}=1$.

Taking into account the equality
\begin{equation} \label{EH_COS}
\frac{x_n}{|y-x|}=(\bs e_{xy}, -\bs e_{n}),
\end{equation}
by (\ref{EH_7AD}) we obtain
\begin{eqnarray*}
{\mathcal C}_1( x; \bs z)
&=&\frac{2}{\omega _n}\sup _{y \in {\mathbb R}^{n-1}}
\frac{|(\bs e_{n}  - n(\bs e_{xy}, \bs e_{n})\bs e_{xy},\; \bs z)|}{x_n^n}
\left (\frac{x_n}{|y-x|}\right )^n \\
& &\\
&=&\frac{2}{\omega _n x_n^n}\sup _{\sigma \in {\mathbb S}^{n-1}_-}
\big |\big (\bs e_n -n(\bs e_{\sigma}, \bs e_n)\bs e_{\sigma},\; \bs z \big )\big |\big (\bs e_{\sigma}, -\bs e_n \big )^n.
\end{eqnarray*}
Replacing here $\bs e_\sigma$ by $-\bs e_\sigma$, we arrive at (\ref{EH_2A})
for $p=1$  with the sharp constant (\ref{EH_2B}).

Let $1<p \leq \infty $. Using (\ref{EH_COS}) and the equality
$$
\frac{1}{|y-x|^{nq}}=\frac{1}{x_n^{nq-n+1}}\left (\frac{x_n}{|y-x|} \right )^{n(q-1)}
\frac{x_n}{|y-x|^n}\;,
$$
and replacing  $q$ by $p/(p-1)$ in (\ref{EH_6A}), we conclude that (\ref{EH_2A}) holds
with the sharp constant
$$
C_p(\bs z)=\frac{2}{\omega _n}
\left \{ \int _ {{\mathbb S}^{n-1}_- }
\big |\big (\bs e_n -n(\bs e_{\sigma}, \bs e_n)\bs e_{\sigma},\; \bs z \big )\big |^{p/(p-1)}
\big (\bs e_{\sigma}, -\bs e_n \big )^{n/(p-1)}\;d\sigma \right \}^{(p-1)/p},
$$
where ${\mathbb S}^{n-1}_- =\{ \sigma \in {\mathbb S}^{n-1}: (\bs e_\sigma , \bs e_n)<0 \}$.
Replacing here $\bs e_\sigma$ by $-\bs e_\sigma$, we arrive at (\ref{EH_3DA})
for $1<p<\infty$ and at (\ref{EH_3H}) for $p=\infty$.

By (\ref{EH_DIR}) we have
$$
\big |\nabla u(x)\big |=\frac{2}{\omega _n}\sup _{|\bs z|=1}\int _{{\mathbb R}^{n-1}}
\frac{(\bs e_{n}  - n(\bs e_{xy}, \bs e_{n})\bs e_{xy},\; \bs z)}{|y-x|^n}\; u(y')dy'.
$$
Hence, by the permutation of suprema, (\ref{EH_6A}), (\ref{EH_7AD}) and (\ref{EH_2A}),
\begin{eqnarray} \label{EH_6AJ}
{\mathcal C}_p( x)&=&\frac{2}{\omega _n}\sup _{|\bs z|=1}
\left \{\int _{{\mathbb R}^{n-1}}
\frac{\big |\big (\bs e_{n}  - n(\bs e_{xy}, \bs e_{n})\bs e_{xy}, \bs z \big )\big |^q }
{|y-x|^{nq}}\;dy' \right \}^{1/q} \nonumber\\
& &\nonumber\\
&=&\sup _{|\bs z|=1}{\mathcal C}_p( x; \bs z)=\sup _{|\bs z|=1}C_p(\bs z) x_n^{(1-n-p)/p}
\end{eqnarray}
for $1<p \leq \infty $, and
\begin{eqnarray} \label{EH_7ADJ}
{\mathcal C}_1( x)&=&\frac{2}{\omega _n}\sup _{|\bs z|=1}\;
\sup _{y \in {\mathbb R}^{n-1}}
\frac{|(\bs e_{n}  - n(\bs e_{xy}, \bs e_{n})\bs e_{xy},\; \bs z)|}{|y-x|^n}\nonumber\\
& &\nonumber\\
&=&\sup _{|\bs z|=1}{\mathcal C}_1( x; \bs z)=\sup _{|\bs z|=1}C_1(\bs z) x_n^{-n}.
\end{eqnarray}
Using the notation (\ref{EH_4J}) in (\ref{EH_6AJ}) and (\ref{EH_7ADJ}),
we arrive at (\ref{EH_3J}).
\end{proof}

\begin{remark}\;{\bf 1}. Formula (\ref{EH_3DA}) for the coefficient $C_p(\bs z), 1<p< \infty,$
can be written with the integral over the whole sphere ${\mathbb S}^{n-1}$
in ${\mathbb R}^{n}$,
$$
C_p(\bs z)\!=\!\frac{2^{1/p}}{\omega _n}
\left \{ \int _ {{\mathbb S}^{n-1} }
\big |\big (\bs e_n\!-\!n(\bs e_{\sigma}, \bs e_n)\bs e_{\sigma},\; \bs z \big )\big |^{p/(p-1)}
\big | \big (\bs e_{\sigma}, \bs e_n \big )\big |^{n/(p-1)}\;d\sigma \right \}^{(p-1)/p}\!.
$$
A similar remark relates (\ref{EH_3H}) as well as
formula (\ref{EH_2B}):
$$
C_1(\bs z)=\frac{2}{\omega _n}\sup _{\sigma \in {\mathbb S}^{n-1}}\big |\big (\bs e_n -n(\bs e_{\sigma}, \bs e_n)\bs e_{\sigma},\; \bs z \big )
\big (\bs e_{\sigma}, \bs e_n \big )^n \big |\;.
$$
\end{remark}

\setcounter{equation}{0}
\section{The case $1\leq p <\infty $}

The next assertion is based on the representation for $C_p$,
obtained in Lemma \ref{L1_1}.

\setcounter{theorem}{0}
\begin{proposition} \label{SP1_1} Let $u \in h^p({\mathbb R}^n_{+})$,
and let  $x $ be  an arbitrary point in ${\mathbb R}^n _+$.
The sharp coefficient ${\mathcal C}_p (x)$ in the inequality
\begin{equation} \label{EH_2_3}
|\nabla u(x)|\leq {\mathcal C}_p( x)\big|\!\big |u \big |\!\big |_p
\end{equation}
is given by
\begin{equation} \label{EH_2A_3A}
{\mathcal C}_p( x)=C_p x_n^{(1-n-p)/p}\;,
\end{equation}
where
\begin{equation} \label{EH_2BAA}
C_1=\frac{2 (n-1)}{\omega _n},
\end{equation}
and
\begin{equation} \label{EH_1ABC}
C_p\!=\!\frac{2 (\omega_{n-2})^{(p-1)/p}}{\omega_n}
\sup _{\gamma \geq 0}\;\frac{1}{\sqrt{1+\gamma^2}}\left \{ \int _ {0}^{\pi}d\varphi\int _ {0}^{\pi/2}{\mathcal F}_{n,p}(\varphi, \vartheta ; \gamma) \;
d\vartheta\right \}^{(p-1)/p},
\end{equation}
if  $1<p<\infty $. Here
\begin{equation} \label{EH_1AC}
{\mathcal F}_{n,p}(\varphi, \vartheta ; \gamma)=\big |(n\cos^2 \vartheta -1)+n\gamma\cos \vartheta
\sin \vartheta\cos \varphi \big |^{p/(p-1)} \cos^{n/(p-1)}\vartheta \sin^{n-2}\vartheta \sin^{n-3}
\varphi\;.
\end{equation}

In particular,
\begin{equation} \label{EH_2AC2}
C_2=\sqrt{\frac{n(n-1)}{2^n\omega_{n}}}\;.
\end{equation}

For $p=1$ and $p=2$ the coefficient ${\mathcal C}_p( x)$ is optimal also in
the weaker inequality obtained from $(\ref{EH_2_3})$ by replacing $\nabla u$ by
$\partial u /\partial x_n$.
\end{proposition}

\begin{proof} The equality (\ref{EH_2A_3A}) was proved in Lemma \ref{L1_1}.

(i) Let $p=1$. Using (\ref{EH_2B}), (\ref{EH_4J}) and the permutability of two
suprema, we find
\begin{eqnarray} \label{EH_2BP}
C_1&=&\frac{2}{\omega _n}\sup_{|\bs z|=1}\sup _{\sigma \in {\mathbb S}^{n-1}_+ }
\big |\big (\bs e_n -n(\bs e_{\sigma}, \bs e_n)\bs e_{\sigma},\; \bs z \big )
\big |\big (\bs e_{\sigma}, \bs e_n \big )^n \nonumber\\
& &\nonumber\\
&=&\frac{2}{\omega _n}\sup _{\sigma \in {\mathbb S}^{n-1}_+ }
\big |\bs e_n -n(\bs e_{\sigma}, \bs e_n)\bs e_{\sigma}\big |
\big (\bs e_{\sigma}, \bs e_n \big )^n \;.
\end{eqnarray}
Taking into account the equality
\begin{eqnarray*}
|\bs e_{n} - n(\bs e_{\sigma}, \bs e_{n})\bs e_{\sigma}|&=&
\Big (\bs e_{n}  - n(\bs e_{\sigma}, \bs e_{n})\bs e_{\sigma}, \;\bs e_{n}
- n(\bs e_{\sigma}, \bs e_{n})\bs e_{\sigma} \Big )^{1/2}\\
& &\\
&=&\Big (1+(n^2-2n)(\bs e_{\sigma},\; \bs e_{n})^2 \Big )^{1/2},
\end{eqnarray*}
and using (\ref{EH_2BP}), we arrive at the sharp constant (\ref{EH_2BAA}).

Furthermore, by (\ref{EH_2B}),
$$
C_1(\bs e_n)=\frac{2}{\omega _n}\sup _{\sigma \in {\mathbb S}^{n-1}_+ }
\big |1 -n(\bs e_{\sigma}, \bs e_n)^2 |\big (\bs e_{\sigma}, \bs e_n \big )^n \geq
\frac{2 (n-1)}{\omega _n}.
$$
Hence, by $C_1 \geq C_1(\bs e_n)$ and by (\ref{EH_2BAA}) we obtain $C_1=C_1(\bs e_n)$,
which completes the proof in the case $p=1$.

\medskip
(ii) Let $1<p<\infty $.
Since the integrand in (\ref{EH_3DA}) does not change when $\bs z \in {\Bbb S}^{n-1}$
is replaced by $-\bs z$, we may assume that $z_n=(\bs e_n, \bs z) > 0$ in (\ref{EH_4J}).

Let $\bs z'=\bs z-z_n\bs e_n$. Then $(\bs z', \bs e_n)=0$ and hence
$z^2_n+|\bs z'|^2=1$. Analogously, with
$\sigma=(\sigma_1,\dots,\sigma_{n-1},\sigma_n) \in {\mathbb S}^{n-1}_+$,
we associate the vector $\bs \sigma '=\bs e_\sigma -\sigma_n\bs e_n$.

Using the equalities $(\bs \sigma ', \bs e_n)=0$,
$\sigma_n =\sqrt{1-|\bs \sigma '|^2}$
and $(\bs z', \bs e_n)=0$, we find an expression for
$(\bs e_n -n(\bs e_{\sigma}, \bs e_n)\bs e_{\sigma},\; \bs z \big )$
as a function of $\bs \sigma'$:
\begin{eqnarray}
(\bs e_n -n(\bs e_{\sigma}, \bs e_n)\bs e_{\sigma},\; \bs z \big )&=&
z_n-n\sigma_n\big ( \bs e_{\sigma}, \bs z \big )=z_n-n\sigma_n
\big ( \bs \sigma '+\sigma_n \bs e_n,\; \bs z'+z_n\bs e_n \big )\nonumber\\
&=&z_n-n\sigma_n
\big [ \big ( \bs \sigma ', \bs z' \big )+z_n \sigma_n\big ]\nonumber\\
&=&-\big [n(1-|\bs\sigma '|^2) -1 \big ]z_n-n\sqrt{1-|\bs\sigma '|^2}\;
\big ( \bs \sigma ', \bs z' \big ) .
\label{EH_B}
\end{eqnarray}

Let ${\mathbb B}^{n-1}=\{ x'=(x_1,\dots,x_{n-1})
\in {\mathbb R}^{n-1}: |x'|< 1 \}$.
By (\ref{EH_B}), taking into account that $d\sigma=d\sigma '/\sqrt{1-|\bs\sigma '|^2}$ ,
we may write (\ref{EH_3DA}) as
\begin{equation} \label{EH_B1C}
C_p=\frac{2}{\omega _n}\sup _{\bs z \in {\Bbb S}^{n-1}_+}\left \{ \int _ {{\mathbb B}^{n-1}}
{\mathcal H}_{n, p}\big  ( |\bs \sigma '|,  (\bs \sigma ',  \bs z') \big  )\;d\sigma ' \right \}^{(p-1)/p},
\end{equation}
where
\begin{eqnarray*}
& &\hspace{-2cm}{\mathcal H}_{n, p}\big  ( |\bs \sigma '|,  (\bs \sigma ',  \bs z') \big  )={\Big | \big [n(1-|\bs\sigma '|^2) -1 \big ]z_n+n
\sqrt{1-|\bs\sigma '|^2}\;\big ( \bs \sigma ', \bs z' \big )\Big |^{p/(p-1)}
\big ( 1-|\bs\sigma '|^2\big )^{n/2(p-1)} \over \sqrt{1-|\bs\sigma '|^2}}\\
& &\\
&=&\Big |\big [n(1-|\bs\sigma '|^2) -1 \big ]z_n+n
\sqrt{1-|\bs\sigma '|^2}\;\big ( \bs \sigma ', \bs z' \big )\Big |^{p/(p-1)}
\big ( 1-|\bs\sigma '|^2\big )^{(n+1-p)/2(p-1)}.
\end{eqnarray*}
Let $B^{n}=\{ x \in {\mathbb R}^{n}: |x|< 1 \}$.
Using the well known formula (see e.g. \cite{PBM}, \textbf{3.3.2(3)}),
$$
\int_{B^n}g\big (|\bs x|, (\bs a, \bs x)\big )dx=\omega_{n-1}\int_0^1 r^{n-1}dr
\int_0^\pi g\big ( r, |\bs a|r \cos \varphi \big )\sin ^{n-2}\varphi \;d\varphi \;,
$$
we obtain
\begin{equation} \label{E_PBM}
\int _ {{\mathbb B}^{n-1}}{\mathcal H}_{n, p}\big  ( |\bs \sigma '|,  (\bs \sigma ',  \bs z') \big  )\;d\sigma'
=\omega_{n-2}\int^{1}_{0} r^{n-2}dr\int^{\pi}_{0}{\mathcal H}_{n, p}\big  (r, \; r|\bs z'|\cos  \varphi  \big  )
\sin ^{n-3}\varphi \;d\varphi\;,
\end{equation}
where
$$
{\mathcal H}_{n, p}\big  (r, \; r|\bs z'|\cos  \varphi  \big  )=\Big |\big [n(1-r^2) -1 \big ]z_n+n\sqrt{1-r^2}\;r|\bs z'|
\cos \varphi \Big |^{p/(p-1)} \big ( 1-r^2\big )^{(n+1-p)/2(p-1)}.
$$
Making the change of variable $r=\sin \vartheta $ in (\ref{E_PBM}), we find
\begin{eqnarray*}
& &\int _ {{\mathbb B}^{n-1}}{\mathcal H}_{n, p}\big  ( |\bs \sigma '|,  (\bs \sigma ',  \bs z') \big  )\;d\sigma'
\\
& &\\
& &=\omega _{n-2}\int^{\pi}_{0}\sin ^{n-3}\varphi d\varphi \int^{\pi/2}_{0}
\Big |\big (n \cos^2 \vartheta -1 \big )z_n+n|\bs z'|\cos \vartheta \sin \vartheta
\cos \varphi \Big |^{p/(p-1)} \sin ^{n-2}\vartheta \cos^{n/(p-1)}\vartheta \;d\vartheta  \;.
\end{eqnarray*}
Introducing here the parameter $\gamma =|\bs z'|/z_n$ and using
the equality $|\bs z'|^2+z^2_n=1$ together with (\ref{EH_B1C}) and (\ref{EH_1AC}),
we arrive at (\ref{EH_1ABC}).

\medskip
(iii) Let $p=2$. By (\ref{EH_1ABC}) and (\ref{EH_1AC}),
\begin{equation} \label{EH_9ABC}
C_2= \frac{2\sqrt{\omega_{n-2}}}{\omega_n}\;\sup _{\gamma \geq 0}\;\frac{1}
{\sqrt{1+\gamma^2}}\left \{ \int _ {0}^{\pi}d\varphi\int _ {0}^{\pi/2}\!\!
{\mathcal F}_{n,2}(\varphi, \vartheta ; \gamma) d\vartheta\right \}^{1/2},
\end{equation}
where
\begin{equation} \label{EH_9ABCD}
{\mathcal F}_{n,2}(\varphi, \vartheta ; \gamma)=\big [(n\cos^2 \vartheta -1)+n\gamma\cos \vartheta
\sin \vartheta\cos \varphi \big ]^{2} \cos^n\vartheta \sin^{n-2}\vartheta \sin ^{n-3}\varphi.
\end{equation}
The equalities (\ref{EH_9ABC}) and (\ref{EH_9ABCD}) imply
\begin{equation} \label{EH_11ABCD}
C_2=\frac{2\sqrt{\omega_{n-2}}}{\omega_n}\sup _{\gamma \geq 0}\;
\frac{1}{\sqrt{1+\gamma^2}}\left \{{\mathcal I}_1+ \gamma^2 {\mathcal I}_2 \right \}^{1/2},
\end{equation}
where
\begin{equation} \label{EH_9B}
{\cal I}_1=\int _ {0}^{\pi}\sin ^{n-3}\varphi\;d\varphi\int _ {0}^{\pi/2}(n\cos^2 \vartheta -1)^2\sin^{n-2}
\vartheta \cos^n \vartheta\;d\vartheta =
{ \sqrt{\pi}\; n(n-1)\Gamma\left ({n-2 \over 2} \right )
\Gamma\left ({n+1 \over 2} \right )\over 8(n-1)!},
\end{equation}
\begin{equation} \label{EH_9BA}
{\cal I}_2=n^2\int _ {0}^{\pi}\cos^2 \varphi \sin ^{n-3}\varphi\; d\varphi\int _ {0}^{\pi/2}\sin^{n}\vartheta \cos^{n+2} \vartheta\; d\vartheta=
{ \sqrt{\pi}\; n\Gamma\left ({n-2 \over 2} \right )\Gamma\left ({n+1 \over 2} \right )
\over 8(n-1)!}.
\end{equation}
By (\ref{EH_11ABCD}) we have
$$
C_2=\!\frac{2\sqrt{\omega_{n-2}}}{\omega_n}\max \big \{ {\cal I}_1^{1/2}, {\cal I}_2^{1/2} \big  \},
$$
which together with (\ref{EH_9B}) and (\ref{EH_9BA}) gives
$$
C_2=\!\frac{2\sqrt{\omega_{n-2}}}{\omega_n}\; {\cal I}_1^{1/2}=
\sqrt{(n-1)\Gamma\left ({n+2 \over 2} \right )\over 2^n \pi^{n/2}}.
$$
Hence (\ref{EH_2AC2}) follows.

Since $\bs z \in {\Bbb S}^{n-1}$ and the supremum in $\gamma =|\bs z'|/z_n$ in
(\ref{EH_9ABC}) is attained for $\gamma=0$, we have $C_2=C_2(\bs e_n)$.
\end{proof}

\setcounter{equation}{0}
\section{An auxilliary algebraic inequality}

Here we prove an algebraic inequality to be
used later for deriving an explicit formula for the sharp constant in the
estimate for the modulus of gradient in the case $p=\infty$.

\setcounter{theorem}{0}
\begin{lemma} \label{Lem_1} For all $x\geq 0$ and any $\mu\geq 1$
the inequality holds
\begin{equation} \label{E_ineq}
\left (\mu+1 \over \mu+x \right  )^{\mu-1}+\left( \mu +1\over 1+\mu x \right )^{\mu-1} x^{\mu+1}
\leq 2x+{\mu(3\mu+1)\over (\mu+1)^2}(1-x)^2 .
\end{equation}
The equality sign takes place only for $\mu=1$ or $x=1$.
\end{lemma}
\begin{proof} Clearly,  the inequality (\ref{E_ineq}) becomes equality for
$\mu=1$ or $x=1$. Suppose that $\mu \in (1, \infty )$.

(i) {\it The case $0\leq x<1$}. Let us write (\ref{E_ineq}) in the form
$$
\left (\mu+1 \over  \mu+x \right  )^{\mu-1}+\left(\mu x+x \over 1+\mu x\right )^{\mu-1} x^2
\leq 2x+{\mu(3\mu+1)\over (\mu+1)^2}(1-x)^2 .
$$
Introducing the notation
\begin{equation} \label{EH_MH1}
F(x)=\left (\mu+1 \over \mu+x \right  )^{\mu-1}+
\left(\mu x+x \over 1+\mu x \right )^{\mu-1} x^2,
\end{equation}
we find
$$
F'(x)=-{\mu-1 \over \mu+1}\left (\mu+1 \over \mu+x \right  )^{\mu}+
x\left ( 2+{\mu-1 \over 1+\mu x} \right )\left(\mu x+x \over 1+\mu x \right )^{\mu-1},
$$
\begin{equation} \label{EH_MH3}
F''(x)={\mu(\mu-1) \over (\mu+1)^2}\left (\mu+1 \over  \mu+x \right  )^{\mu+1}+
\mu\left [ {2(x+1) \over 1+\mu x} +{\mu-1 \over (1+\mu x)^2}\right ]
\left(\mu x+x \over 1+\mu x \right )^{\mu-1},
\end{equation}
\begin{equation} \label{EH_MH4}
F'''(x)=-{\mu(\mu-1) \over (\mu+1)^2}\left (\mu+1 \over  \mu+x \right  )^{\mu+2}+
{\mu(\mu^2-1) \over x(1+\mu x)^3}\left(\mu x+x \over 1+\mu x \right )^{\mu-1}.
\end{equation}

By Taylor's formula with Lagrange's remainder term,
\begin{equation} \label{EH_MH5}
F(x)=F(1)+F'(1)(x-1)+{1 \over 2}F''(t)\;(x-1)^2=2+2(x-1)+
{1 \over 2}F''(t)\;(x-1)^2,
\end{equation}
where $x \in [0, 1)$ and $t \in (x, 1)$.

Note that $F''(0)<F''(1)$. In fact, by (\ref{EH_MH3}),
$$
F''(0) ={\mu-1 \over \mu+1}\left (1+{1 \over \mu} \right )^\mu,\;\;\;\;\;
F''(1) ={2\mu(3\mu+1) \over (\mu+1)^2},
$$
which together with the obvious inequality
$$
{\mu-1 \over \mu+1}\;e < {2\mu(3\mu+1) \over (\mu+1)^2}
$$
implies $F''(0)<F''(1)$.

Next we show that
\begin{equation} \label{EH_MH6}
F''(t) <\max \big \{ F''(0),\; F''(1)\big \} ={2\mu(3\mu+1) \over (\mu+1)^2}
\end{equation}
for any $t \in (0, 1)$.

Suppose the opposite assertion holds, i.e. there exists a point  $t \in (0, 1)$
at which (\ref{EH_MH6}) fails. Hence $F''(t)$ attains its maximum value on $[0, 1]$
at an inner point $\tau $, i.e.,
\begin{equation} \label{EH_MH7}
F''(\tau) =\max _{t \in [0; 1]}F''(t)\geq \max \big
\{ F''(0),\; F''(1)\big \}={2\mu(3\mu+1) \over (\mu+1)^2}.
\end{equation}
Taking into account (\ref{EH_MH4}), we have
$$
F'''(\tau )=-{\mu(\mu-1) \over (\mu+1)^2}\left (\mu+1 \over  \mu+\tau \right  )^{\mu+2}+
 {\mu(\mu^2-1) \over \tau(1+\mu\tau)^3}\left( \mu\tau+\tau \over 1+\mu\tau \right )^{\mu-1}=0,
$$
which is equivalent to
$$
\left( \mu\tau+\tau \over 1+\mu\tau \right )^{\mu-1}={\tau(1+\mu\tau)^3  \over (\mu+1)^3}
\left (\mu+1 \over \mu+\tau \right  )^{\mu+2}.
$$
Combined with (\ref{EH_MH3}), this implies
\begin{eqnarray*}
F''(\tau )&=&{\mu(\mu-1) \over (\mu+1)^2}\left (\mu+1 \over  \mu+\tau \right  )^{\mu+1}
\left \{1+ \tau(1+\mu\tau){2(1+\tau)(1+\mu\tau)+\mu-1 \over (\mu-1)(\mu+\tau)}\right \}\\
&<&{\mu(\mu-1) \over (\mu+1)^2}\left (\mu+1 \over  \mu+\tau \right  )^{\mu+1}
\left \{2+ {2(1+\tau)(1+\mu\tau) \over \mu-1}\right \}.
\end{eqnarray*}
Therefore,
\begin{eqnarray} \label{EH_MH9}
F''(\tau )&<&{2\mu \over (\mu+1)^2}\left ( \mu+1 \over  \mu+\tau \right  )^{\mu+1}
\Big \{(\mu-1)+(1+\tau)(1+\mu\tau) \Big \}\nonumber\\
& &\nonumber \\
&=&2\mu (\mu+1)^{\mu-1}\;
{\mu\tau^2+(\mu+1)\tau +\mu \over (\mu+\tau)^{\mu+1}}.
\end{eqnarray}
Setting
\begin{equation} \label{EH_MH10}
\eta(\tau)={\mu\tau^2 +(\mu+1)\tau+\mu \over (\mu+\tau)^{\mu+1}},
\end{equation}
we rewrite (\ref{EH_MH9}) as
\begin{equation} \label{EH_MH11}
F''(\tau )<2\mu (\mu+1)^{\mu-1}\eta(\tau ).
\end{equation}
Noting that
$$
\eta '(\tau )={\mu\tau(\mu-1)(1-\tau) \over (\mu+\tau)^{\mu+2}}>0
$$
for $0<\tau <1$, by (\ref{EH_MH10}) and (\ref{EH_MH11}), we find
$$
\max _{t \in [0; 1]}F''(t)=F''(\tau )<2\mu (\mu+1)^{\mu-1}\eta(1)
={2\mu(3\mu+1) \over (\mu+1)^2}.
$$
The latter contradicts (\ref{EH_MH7}) which proves (\ref{EH_MH6}) for all
$t \in (0, 1 )$. Thus, it follows from (\ref{EH_MH1}), (\ref{EH_MH5})
and (\ref{EH_MH6}) that
$$
\left (\mu +1 \over  \mu+x \right  )^{\mu-1}+\left( \mu x+x \over 1+\mu x \right )^{\mu-1} x^2
<2x+{\mu(3\mu+1) \over (\mu+1)^2}\;(x-1)^2
$$
for all $0\leq x<1$. This means that for $\mu>1$ and $0\leq x<1$ the
strict inequality (\ref{E_ineq}) holds.

(ii) {\it The case  $x>1$}. Since the function
$$
G(x)=\left (\mu+1 \over  \mu+x \right  )^{\mu-1}+
\left( \mu x+x \over 1+\mu x \right )^{\mu-1} x^2
-2x-{\mu(3\mu+1) \over (\mu+1)^2}\;(x-1)^2
$$
satisfies the equality
$$
G \left ({1 \over x} \right )={1 \over x^2}\; G(x),
$$
we have by part (i) that $G(x)< 0$ for $0\leq x< 1$ and hence $G(x)< 0$ for $x >1$.
Thus, the strict inequality (\ref{E_ineq}) holds for $\mu>1$ and $x>1$.
\end{proof}

\setcounter{theorem}{0}
\begin{corollary} \label{Cor_1} For all $y\geq 0$ and any natural $n\geq 2$
the inequality holds
\begin{equation} \label{EH_M13}
{\cal P}_n^2(y)+{\cal P}_n^2(-y) \leq {2n^2+4(n-1)(3n-2)y^2 \over n^n},
\end{equation}
where
\begin{equation} \label{EH_M14}
{\cal P}_n(y)={\left ( \sqrt{1+y^2}+y\right )^{n-1} \over \left ( 1+(n-1)
\Big (\sqrt{1+y^2 }+y\Big )^2\right )^{(n-2)/2}}.
\end{equation}
\end{corollary}
\begin{proof}
Suppose that $0<x \leq 1$. We introduce the new variable
\begin{equation} \label{EH_M14A}
y={1-x \over 2\sqrt{x}} \in [0, \infty).
\end{equation}
Solving the equation $y=(x^{-1/2}-x^{1/2})/2$ in $\sqrt{x}$,
we find $\sqrt{x}=\sqrt{1+y^2}-y$, i.e.,
$$
x=\Big (\sqrt{1+y^2 }-y\Big )^2={1 \over \Big (\sqrt{1+y^2 }+y\Big )^2},
$$

Putting $\mu=n-1$, we write (\ref{E_ineq}) as
\begin{equation} \label{EH_M15A}
{1\over \big (n-1+x \big )^{n-2}}+{x^n \over \big(1+(n-1 )x \big )^{n-2}}
\leq {2n^2x+(n-1)(3n-2)(1-x)^2 \over n^n}.
\end{equation}
By (\ref{EH_M14A}) we have $(1-x)^2=4y^2x$, hence (\ref{EH_M15A}) can be rewritten
in the form
\begin{equation} \label{EH_M16A}
{1\over x\big (n-1+x \big )^{n-2}}+{x^{n-1} \over \big(1+(n-1 )x \big )^{n-2}}
\leq {2n^2+4(n-1)(3n-2)y^2 \over n^n}.
\end{equation}
Setting $x=\Big (\sqrt{1+y^2 }+y\Big )^{-2}$ in the first term on the
left-hand side of (\ref{EH_M16A}), we obtain
\begin{equation} \label{EH_M17A}
{1\over x\big (n-1+x \big )^{n-2}}={\left ( \sqrt{1+y^2}+y\right )^{2n-2} \over \left ( 1+(n-1)
\Big (\sqrt{1+y^2 }+y\Big )^{2}\right )^{n-2}}={\cal P}_n^2(y).
\end{equation}
Similarly, putting $x=\Big (\sqrt{1+y^2 }-y\Big )^{2}$ in the second term on
the left-hand side of (\ref{EH_M16A}), we find
\begin{equation} \label{EH_M18A}
{x^{n-1} \over \big(1+(n-1 )x \big )^{n-2}}={\left ( \sqrt{1+y^2}-y\right )^{2n-2} \over \left ( 1+(n-1)
\Big (\sqrt{1+y^2 }-y\Big )^{2}\right )^{n-2}}={\cal P}_n^2(-y).
\end{equation}
Using (\ref{EH_M17A}) and (\ref{EH_M18A}), we can rewrite (\ref{EH_M16A})
as (\ref{EH_M13}).
\end{proof}

We give one more corollary of Lemma \ref{Lem_1} containing an alternative form of (\ref{E_ineq})
with natural $\mu \geq 2$. However, we are not going to use it henceforth.
\begin{corollary} \label{Cor_2} For all $x\geq 0$ and any natural $n\geq 2$
the inequality holds
\begin{equation} \label{EH_RA}
\sum_{k=3}^{n+1}\Big ( \begin{array}{c}n+1\\k \end{array}\Big ) \left \{
{1 \over \big (n+x \big )^{k-2}}+{(-1)^k \over \big (1+nx \big )^{k-2}}\right \}(1-x)^k\leq 0\;.
\end{equation}
The equality sign takes place only for $x=1$.
\end{corollary}
\begin{proof}
We set $\mu=n, n \geq 2$, in (\ref{E_ineq}):
$$
\left (n+1 \over n+x \right  )^{n-1}+\left(n+1\over 1+nx \right )^{n-1} x^{n+1}
\leq 2x+{n(3n+1)\over (n+1)^2}(1-x)^2 .
$$
Multiplying the last inequality by $(n+1)^2$, we write it as
\begin{equation} \label{EH_MA}
{(n+1)^{n+1} \over \big ( n+x \big )^{n-1}}+{\big( (n+1)x \big )^{n+1}
\over \big (1+nx \big )^{n-1}}-2(n+1)^2 x-n(3n+1)(1-x)^2 \leq 0.
\end{equation}

We rewrite the first term in (\ref{EH_MA}):
\begin{eqnarray} \label{EQ_2}
& &{(n+1)^{n+1} \over \big ( n+x \big )^{n-1}}={[(n+x)+(1-x)]^{n+1} \over \big ( n+x\big )^{n-1}}=
{1 \over \big ( n+x\big )^{n-1}}\sum_{k=0}^{n+1}
\Big ( \begin{array}{c}n+1\\k \end{array}\Big ) (n+x)^{n+1-k}(1-x)^k \nonumber\\
& &=(n+x)^2+(n+1)(n+x)(1-x) +{n(n+1) \over 2}(1-x)^2 +
\sum_{k=3}^{n+1}\Big ( \begin{array}{c}n+1\\k \end{array}\Big ){(1-x)^k \over (n+x)^{k-2}}.
\end{eqnarray}
Similarly, the second term in (\ref{EH_MA}) can be written as
\begin{eqnarray} \label{EQ_3}
& &\hspace{-5mm}{\big( (n+1)x \big )^{n+1} \over \big (1+nx \big )^{n-1}}=
{[(1+nx)-(1-x)]^{n+1} \over \big (1+nx \big )^{n-1}}=
{1 \over \big (1+nx \big )^{n-1}}\sum_{k=0}^{n+1}(-1)^k
\Big ( \begin{array}{c}n+1\\k \end{array}\Big )
\big (1+nx \big )^{n+1-k}(1-x)^k \nonumber\\
& &\hspace{-5mm}=\big (1+nx \big )^2-(n+1)\big (1+nx \big )(1-x) +{n(n+1) \over 2}(1-x)^2
+\sum_{k=3}^{n+1}(-1)^k
\Big ( \begin{array}{c}n+1\\k \end{array}\Big ){(1-x)^k \over \big (1+nx \big )^{k-2}}.
\end{eqnarray}

Using (\ref{EQ_2}) and (\ref{EQ_3}) in (\ref{EH_MA}), we arrive at (\ref{EH_RA})
with the left-hand side as the sum of rational functions.
\end{proof}

\setcounter{equation}{0}
\section{The case $p=\infty$}

The next assertion is the main theorem of this paper.
It is based on the representation for the sharp constant $C_p$ ($1<p<\infty$)
obtained in Proposition \ref{SP1_1}. To find the explicit formula for $C_\infty$
we solve an extremal problem with a scalar parameter entering the integrand
in a double integral.

\setcounter{theorem}{0}
\begin{theorem} \label{SP_2} Let $u \in h^\infty({\mathbb R}^n_{+})$,
and let  $x $ be  an arbitrary point in ${\mathbb R}^n _+$.
The sharp coefficient ${\mathcal C}_p (x)$ in the inequality
\begin{equation} \label{EH_X}
|\nabla u(x)|\leq {\mathcal C}_\infty( x)\big|\!\big |u \big |\!\big |_\infty
\end{equation}
is given by
\begin{equation} \label{EH_Y}
{\mathcal C}_\infty( x)=\frac{4(n-1)^{(n-1)/2}\;\omega_{n-1}}{ n^{n/2}\;\omega_n x_n}\;.
\end{equation}

For $p=\infty$ the absolute value of the derivative of a harmonic function $u$
with respect to the normal to the boundary of the half-space at any
$x \in {\Bbb R}^n_+$ has the same supremum as $|\nabla u(x)|$.
\end{theorem}

\begin{proof} We pass to the limit as $p \rightarrow \infty $ in (\ref{EH_2_3}),
(\ref{EH_2A_3A}), (\ref{EH_1ABC}) and (\ref{EH_1AC}). This results in
\begin{equation} \label{EH_inf}
{\cal C}_\infty (x)=C_\infty \;x_n^{-1}\;,
\end{equation}
where
\begin{equation} \label{EH_inf1}
C_\infty=\sup _{\gamma \geq 0}\;
{2 \omega _{n-2} \over \omega _n\sqrt{1+\gamma^2}}\int _ {0}^{\pi}
\sin ^{n-3}\varphi\;d\varphi\int _ {0}^{\pi/2}
\;  \big |(n\cos^2 \vartheta -1)+n\gamma\cos \vartheta
\sin \vartheta\cos \varphi \big | \sin^{n-2}\vartheta\;d\vartheta.
\end{equation}

We are looking for a solution of the equation
\begin{equation} \label{EH_inf2}
(n\cos^2 \vartheta -1)+n\gamma\cos \vartheta \sin \vartheta\cos \varphi =0
\end{equation}
as a  function $\vartheta$ of $\varphi$. We can rewrite  (\ref{EH_inf2}) as the
second order equation in $\tan \vartheta$:
$$
\tan^2\vartheta -n\gamma \cos \varphi\tan \vartheta  +1-n=0.
$$
Since $0\leq \vartheta \leq \pi/2$, we find that the nonnegative root of this equation is
\begin{equation} \label{EH_inf3}
\vartheta _\gamma(\varphi)=\arctan \left ( {n\gamma \cos \varphi + \sqrt{4(n-1)+n^2\gamma ^2\cos ^2\varphi}\over 2}\right ).
\end{equation}
Taking into account that the left-hand side of (\ref{EH_inf2}) is nonnegative for
$0\leq \vartheta\leq \vartheta _\gamma(\varphi),\; 0\leq \varphi \leq \pi$, and
using the equalities
$$
\int _ {0}^{\vartheta}\big [(n\cos^2 \vartheta -1)+n\gamma\cos \varphi\cos \vartheta
\sin \vartheta \big ] \sin^{n-2}\vartheta\;d\vartheta=[\cos \vartheta + \gamma
\cos \varphi\sin \vartheta ]\sin^{n-1}\vartheta \;,
$$
$$
\int _ {0}^\pi\sin ^{n-3}\varphi\;d\varphi\int _ {0}^{\pi/2} \;\big [(n\cos^2 \vartheta -1)+n\gamma\cos \varphi\cos \vartheta
\sin \vartheta \big ] \sin^{n-2}\vartheta  d\vartheta=
\gamma\int_0^\pi\sin^{n-3}\varphi \cos \varphi d\varphi=0,
$$
we write (\ref{EH_inf1}) as
\begin{eqnarray} \label{EH_inf4}
C_\infty&=&\sup _{\gamma \geq 0}\;
{4 \omega _{n-2} \over \omega _n\sqrt{1+\gamma^2}}\int _ {0}^{\pi}\sin ^{n-3}\varphi \;d\varphi
\int _ {0}^{\vartheta_\gamma(\varphi )}
\;  \big [(n\cos^2 \vartheta -1)+n\gamma\cos \varphi\cos \vartheta
\sin \vartheta \big ] \sin^{n-2}\vartheta  d\vartheta\nonumber\\
& &\nonumber\\
&=&
\sup _{\gamma \geq 0}\;{4 \omega _{n-2} \over \omega _n\sqrt{1+\gamma^2}}\int _ {0}^{\pi}\;
\big [\cos \vartheta_\gamma(\varphi )+\gamma\cos \varphi
\sin \vartheta _\gamma(\varphi) \big ] \sin^{n-1}\vartheta _\gamma(\varphi)\sin ^{n-3}\varphi\;d\varphi .
\end{eqnarray}

By (\ref{EH_inf3}),
\begin{equation} \label{EH_sin}
\sin \vartheta _\gamma(\varphi)={n\gamma \cos \varphi + \sqrt{4(n-1)+n^2\gamma ^2\cos ^2\varphi}
\over \sqrt{4+\Big ( n\gamma \cos \varphi + \sqrt{4(n-1)+n^2\gamma ^2\cos ^2\varphi }\;\Big )^2}}\; ,
\end{equation}
\begin{equation} \label{EH_cos}
\cos \vartheta _\gamma(\varphi)={2 \over \sqrt{4+\Big ( n\gamma \cos \varphi + \sqrt{4(n-1)+n^2
\gamma ^2\cos ^2\varphi }\;\Big )^2} }\; .
\end{equation}
By (\ref{EH_sin}) and (\ref{EH_cos}), we find
\begin{equation} \label{EH_comb}
\cos \vartheta_\gamma(\varphi )+\gamma \cos \varphi \sin \vartheta _\gamma(\varphi)=
{2+ \gamma \cos \varphi \left (n\gamma \cos \varphi + \sqrt{4(n-1)+
n^2\gamma ^2\cos ^2\varphi }\right ) \over \sqrt{4+\left ( n\gamma \cos \varphi + \sqrt{4(n-1)+n^2\gamma ^2\cos ^2\varphi }\;\right )^2}}\;.
\end{equation}
Using (\ref{EH_sin}), (\ref{EH_comb}), and the identity
$$
2\!+\! \gamma \cos \varphi \left (n\gamma \cos \varphi \!+\! \sqrt{4(n\!-\!1)\!+\!
n^2\gamma ^2\cos ^2\varphi }\right )={4\!+\!\left ( n\gamma \cos \varphi \!+\! \sqrt{4(n\!-\!1)\!+\!n^2\gamma ^2\cos ^2\varphi }\;\right )^2 \over 2n}\;,
$$
we can write (\ref{EH_inf4}) as
$$
C_\infty=\sup _{\gamma \geq 0}\;
{2 \omega _{n-2} \over n\omega _{n}\sqrt{1+\gamma^2}}\int _ {0}^{\pi}
{\left (n\gamma \cos \varphi + \sqrt{4(n-1)+n^2\gamma ^2\cos^2\varphi}\right )^{n-1}
\over \left ( 4+\Big ( n\gamma \cos \varphi + \sqrt{4(n-1)+n^2
\gamma ^2\cos ^2\varphi }\Big )^2\right )^{(n-2)/2}}\;\sin^{n-3}\varphi\;d\varphi.
$$
Introducing the parameter
$$
\alpha={n\gamma \over 2\sqrt{n-1} },
$$
we obtain
\begin{equation} \label{EH_inf6}
C_\infty=\sup _{\alpha \geq 0}\;{4\omega _{n-2}(n-1)^{(n-1)/2} \over \omega_{n}\sqrt{n^2+4(n-1)\alpha^2}}
\int _ {0}^{\pi}{\cal P}_n\big ( \alpha \cos \varphi \big )\sin^{n-3}\varphi\;d\varphi\;,
\end{equation}
with ${\cal P}_n$ defined by (\ref{EH_M14}).
The change of variable $t=\cos \varphi$ in (\ref{EH_inf6}) implies
\begin{equation} \label{EH_inf7}
C_\infty=\sup _{\alpha \geq 0}\;
{4\omega _{n-2}(n-1)^{(n-1)/2} \over \omega _{n}\sqrt{n^2+4(n-1)\alpha^2}}\int _ {-1}^{1}
{\cal P}_n(\alpha t)(1-t^2)^{(n-4)/2}\;dt.
\end{equation}
Integrating in (\ref{EH_inf7}) over $(-1, 0)$ and $(0, 1)$, we have
$$
C_\infty=\sup _{\alpha \geq 0}\;
{4\omega _{n-2}(n-1)^{(n-1)/2} \over \omega _{n}\sqrt{n^2+4(n-1)\alpha^2}}\int _ {0}^{1}
{ {\cal P}_n(\alpha t) + {\cal P}_n(-\alpha t) \over (1-t^2)^{(4-n)/2}}\;dt.
$$

Applying the Schwarz inequality, we see that
$$
C_\infty \leq \sup _{\alpha \geq 0}\;
{4\omega _{n-2}(n-1)^{(n-1)/2} \over \omega _{n}\sqrt{n^2+4(n-1)\alpha^2}}\;\left \{ \int _ {0}^{1}
\frac{\big ( {\mathcal P}_n(\alpha t) \!+\!{\mathcal P}_n(-\alpha t)\big )^2 }
{(1-t^2)^{(4-n)/2}}\;dt \right \}^{1/2}\!\!
\left \{ \int _ {0}^{1}\!{dt \over(1-t^2)^{(4-n)/2}}\; \right \}^{1/2}.
$$
Hence, taking into account the identity
\begin{equation} \label{EH_IN1}
\int _ {0}^{1}\!(1-t^2)^{(n-4)/2}\;dt=
{ \sqrt{\pi}\;\Gamma\left ({n-2 \over 2} \right )
\over 2\Gamma\left ({n-1 \over 2} \right )},
\end{equation}
we obtain
\begin{equation} \label{EH_INEQ}
C_\infty \leq \sup _{\alpha \geq 0}\;
{4\omega _{n-2}(n-1)^{(n-1)/2} \over \omega _{n}\sqrt{n^2+4(n-1)\alpha^2}}\;\left ( { \sqrt{\pi}\;\Gamma\left ({n-2 \over 2} \right )
\over 2\Gamma\left ({n-1 \over 2} \right )}\right )^{1/2}\left \{ \int _ {0}^{1}
\frac{\big ( {\mathcal P}_n(\alpha t) \!+\!{\mathcal P}_n(-\alpha t)\big )^2 }
{(1-t^2)^{(4-n)/2}}\;dt \right \}^{1/2}.
\end{equation}
By (\ref{EH_M14}),
$$
{\mathcal P}_n(y){\mathcal P}_n(-y)=\Big (4(n-1)y^2 +n^2 \Big )^{(2-n)/2},
$$
which implies
\begin{equation} \label{EH_P}
{\mathcal P}_n(y){\mathcal P}_n(-y)\leq  n^{2-n}.
\end{equation}
Combining  (\ref{EH_P}) and (\ref{EH_M13}), we obtain
$$
\big ({\mathcal P}_n(y)+{\mathcal P}_n(-y)\big )^2 \leq
\;{\mathcal P}_n^2(y)+{\mathcal P}_n^2(-y)+2 n^{2-n} \leq
{2n^2+4(n-1)(3n-2)y^2 \over n^n}+{2 \over n^{n-2}}.
$$
Therefore,
\begin{equation} \label{EH_INEQ1}
\big ({\mathcal P}_n(\alpha t)+{\mathcal P}_n(-\alpha t)\big )^2 \leq
{4 \over n^{n-2}}\left (1+{(n-1)(3n-2) \over n^2}\;\alpha^2 t^2 \right ).
\end{equation}

By (\ref{EH_IN1}), (\ref{EH_INEQ}), (\ref{EH_INEQ1}) and by
$$
\int _ {0}^{1}\!t^2(1-t^2)^{(n-4)/2}\;dt=
{ \sqrt{\pi}\;\Gamma\left ({n-2 \over 2} \right )
\over 2(n-1)\Gamma\left ({n-1 \over 2} \right )},
$$
we find
\begin{equation} \label{EH_INEQ2}
C_\infty \leq {4\omega _{n-2}\sqrt{\pi}\;(n-1)^{(n-1)/2}\Gamma\left ({n-2 \over 2} \right ) \over \omega _{n}n^{n/2}\Gamma\left ({n-1 \over 2} \right )}\;\sup _{\alpha \geq 0}\;
\left ( {n^2+(3n-2)\alpha^2 \over n^2+4(n-1)\alpha^2}\right )^{1/2}.
\end{equation}
Note that
$$
{d\over d \alpha}\left ({n^2+(3n-2)\alpha^2 \over n^2+4(n-1)\alpha^2} \right )=
-\frac{2 \alpha(n-2)n^2}{\big ( n^2+4(n-1)\alpha^2 \big )^2}<0 \;\;\mbox{for}\;\; \alpha>0,
$$
therefore the supremum in $\alpha $ on the right-hand side of (\ref{EH_INEQ2})
is attained for $\alpha=0$. Thus,
\begin{equation} \label{EH_inf16}
C_\infty \leq {4\omega _{n-2}\sqrt{\pi}\;(n-1)^{(n-1)/2}
\Gamma\left ({n-2 \over 2} \right ) \over \omega _nn^{n/2}\Gamma\left ({n-1 \over 2} \right )}
= \frac{4(n-1)^{(n-1)/2}\;\omega_{n-1}}{ n^{n/2}\;\omega_n}.
\end{equation}

Besides, in view of (\ref{EH_inf6}) and (\ref{EH_M14}),
\begin{eqnarray*}
C_\infty &\geq& {4\omega _{n-2}(n-1)^{(n-1)/2} \over n\omega _n}
\int _ {0}^{\pi}{\cal P}_n(0)\sin^{n-3}\varphi\;d\varphi
={4\omega _{n-2}(n-1)^{(n-1)/2} \over n\omega _n}
\int _ {0}^{\pi}{\sin^{n-3}\varphi\over n^{(n-2)/2}}\;d\varphi\\
& &\\
&=&{4 \omega _{n-2}\sqrt{\pi}\;(n-1)^{(n-1)/2}
\Gamma\left ({n-2 \over 2} \right ) \over \omega _nn^{n/2}\Gamma\left ({n-1 \over 2} \right )}
= \frac{4(n-1)^{(n-1)/2}\;\omega_{n-1}}{ n^{n/2}\;\omega_n},
\end{eqnarray*}
which together with (\ref{EH_inf}) and (\ref{EH_inf16}) proves the equality (\ref{EH_Y}).

Since $\bs z \in {\Bbb S}^{n-1}$ and the supremum with respect to the parameter
$$
\alpha={n\gamma \over 2\sqrt{n-1} }={n |\bs z'| \over 2 z_n \sqrt{n-1} },
$$
in (\ref{EH_inf6}) is attained for $\alpha=0$, it follows that the absolute value of
the directional derivative of a harmonic function $u$ with respect to the normal
$\bs e_n$ to $\partial {\Bbb R}^n_+$, taken at an arbitrary point $x \in {\Bbb R}^n_+$,
has the same supremum as $|\nabla u(x)|$.
\end{proof}

\begin{remark}\;{\bf 2}. Inequality (\ref{EH_X}) can be written in the form
\begin{equation} \label{EH_2YCA}
|\nabla u(x)|\leq {\cal C}_\infty(x)\sup _{y \in {\mathbb R}^{n}_+}\big | u (y) \big |.
\end{equation}
Using here (\ref{EH_Y}), we arrive at the explicit sharp inequality
$$
|\nabla u(x)|\leq \frac{4(n-1)^{(n-1)/2}\;\omega_{n-1}}{ n^{n/2}\;\omega_n x_n}\;\sup _{y \in {\mathbb R}^{n}_+}\big | u (y) \big |\;,
$$
which generalizes the real value analog (\ref{L_H}) of (\ref{E_KM2})
to harmonic functions in the $n$-dimensional half-space.

From (\ref{EH_2YCA}) it follows that
\begin{equation} \label{EH_2YC}
|\nabla u(x)|\leq {\cal C}_\infty(x)\sup _{y \in {\mathbb R}^{n}_+}\big | u (y) -\omega \big |
\end{equation}
with an arbitrary constant $\omega$. Minimizing (\ref{EH_2YC}) in $\omega $, we obtain
\begin{equation} \label{EH_2YO}
|\nabla u(x)|\leq {{\cal C}_\infty(x) \over 2}\;{\rm osc}_{_{{\mathbb R}^{n}_+}}\!(u),
\end{equation}
where ${\rm osc}_{_{{\mathbb R}^{n}_+}}\!(u)$
is the oscillation of $u$ on ${\mathbb R}^{n}_+$.

Equality \ref{EH_Y}) and inequality (\ref{EH_2YO}) imply the sharp estimate
$$
|\nabla u(x)|\leq \frac{2(n-1)^{(n-1)/2}\;\omega_{n-1}}{ n^{n/2}\;\omega_n x_n}\;
{\rm osc}_{_{{\mathbb R}^{n}_+}}\!(u),
$$
which is an analogue of (\ref{E_KM2A}) for harmonic functions in ${\mathbb R}^{n}_+$.
\end{remark}

\medskip
{\bf Acknowledgments.}
The research of the first author was supported by the KAMEA program of
the Ministry of Absorption, State of Israel, and by the Ariel University
Center of Samaria. The second author was partially supported by the UK
Engineering and Physical Sciences Research Council grant EP/F005563/1.


\bigskip
\bigskip
{
\sc Gershon Kresin}
\\
{\sc e-mail:\;}{\tt kresin@ariel.ac.il
}
\\
{\sc address:\;}{\it Department of Computer Science and Mathematics,
            Ariel University Center of Samaria,\\
            44837 Ariel,
            Israel}
\\
\\
{\sc Vladimir Maz'ya}
\\
{\sc e-mail:\;}{\tt vlmaz@liv.ac.uk}
\\
{\sc address:\;}{\it {Department of Mathematical Sciences,
            University of Liverpool,
            M$\&$O Building,\\ Liverpool, L69 3BX,
            UK}
\\
{\sc e-mail:\;}{\tt vlmaz@mai.liu.se}
\\
{\sc address:\;}{\it Department of Mathematics,
            Link\"oping University,
            SE-58183 Link\"oping,
            Sweden}

\begin{thebibliography}{99}

\bibitem{Bo1} E. Borel, \textit{D\'emonstration \'el\'ementaire d'un th\'eor\`eme
de M. Picard sur les fonctions enti\`eres}.
C.R. Acad. Sci. \textbf{122} (1896),1045--1048.

\bibitem{Bo2} E. Borel, \textit{M\'ethodes et Probl\`emes de Th\'eorie des Fonctions}.
Gauthier-Villars, Paris, 1922.

\bibitem{Bur} R.B. Burckel,  \textit{An Introduction to Classical Complex Analysis, V. 1}.
Academic Press, New York-San Francisco, 1979.


\bibitem{Ha} J. Hadamard, \textit{Sur les fonctions enti\`eres de la forme $e^{G(X)}$}.
C.R. Acad. Sci. \textbf{114} (1892), 1053--1055.

\bibitem{KHAV}D. Khavinson, An extremal problem for harmonic functions in the ball, {\it Canad.  Math. Bull.,} {\bf 35}(2), 1992, 218-220.

\bibitem{KM} G. Kresin and V. Maz'ya,  \textit{Sharp Real-Part Theorems. A Unified Approach}.
Lect. Notes in Math., \textbf{1903}, Springer-Verlag, Berlin-Heidelberg-New York, 2007.

\bibitem{KM_1} G. Kresin and V. Maz'ya,  \textit{Sharp pointwise estimates for solutions of
strongly elliptic second order systems with boundary data from $L^p$}. Appl. Anal.
\textbf{86, N. 7} (2007), 783--805.

\bibitem{Lan1} E. Landau, \textit{\"Uber den Picardschen Satz}.
Vierteljahrschr. Naturforsch. Gesell. Z\"urich \textbf{51} (1906), 252--318.

\bibitem{Lan2} E. Landau, \textit{Beitr\"age zur analytischen Zahlentheorie}.
Rend. Circ. Mat. Palermo \textbf{26} (1908), 169--302.

\bibitem{Lin} E. Lindel\"of, \textit{M\'emoire sur certaines
in\'egalit\'es dans la th\'eorie des fonctions monog\`enes et
sur quelques propri\'et\'es nouvelles de ces fonctions dans
le voisinage d'un point singulier essentiel}.
Acta Soc. Sci. Fennicae \textbf{35, N. 7} (1908), 1--35.

\bibitem{PBM} A. P. Prudnikov, Yu. A. Brychkov and O. I. Marichev,
\textit{Integrals and Series, Vol. 1, Elementary Functions}.
Gordon and Breach Sci. Publ., New York, 1986.

\end{thebibliography}
\end{document}